\documentclass[11pt]{article}
\usepackage{amsmath, amssymb, amscd, amsthm, amsfonts}
\usepackage{graphicx}
\usepackage{hyperref}
\usepackage{authblk}
\usepackage{stmaryrd} 
\usepackage{tikz}
\usetikzlibrary{arrows.meta,decorations.pathmorphing,calc}

\title{A Sudakov Decomposition in Riemannian Manifolds with Positive Curvature}
\author{Zhengyao Huang}
\affil{Department of Mathematics, Durham University \\ \texttt{phqv76@durham.ac.uk}}

\newtheorem{theorem}{Theorem}
\newtheorem{maintheorem}{Theorem}
\newtheorem{lemma}[theorem]{Lemma}
\newtheorem{corollary}[theorem]{Corollary}

\newtheorem{defn}[theorem]{Definition}

\theoremstyle{remark}
\newtheorem{remark}[theorem]{Remark}

\begin{document}

\maketitle

\begin{abstract}
In this paper, we study Monge's problem on Riemannian manifolds $(M, g)$ with positive sectional curvature. Assuming that the source and target measures are absolutely continuous with respect to the Riemannian volume measure, we generalize a variational method from the Euclidean setting to establish the existence of a transport density and an explicit disintegration of measures along optimal rays. These results extend the approach of Bianchini-Caravenna to the Riemannian context.
\end{abstract}

\tableofcontents
\newpage

\section{Introduction}\label{section-introduction}
We study Monge's problem on Riemannian manifolds $(M,g)$ with positive sectional curvature, i.e.\ $\operatorname{Sec}(M) \geq K > 0$. We further assume that $\mu$ and $\nu$ are probability measures that are absolutely continuous with respect to the Riemannian volume measure $d\operatorname{Vol}_g$. Our focus is on the structure of the transport set and its disintegration along transport rays. The strategy we follow adapts a variational method introduced by Bianchini~\cite{bianchini2007eulerlagrange} and later refined by Caravenna~\cite{caravenna2011a}. Related disintegration techniques have also been used to study the graph of convex functions~\cite{caravenna2010a} and optimal transport with norm costs~\cite{2013arXiv1311.1918B}.

In the Euclidean setting for the $L^1$ cost $c(x,y)=\|x-y\|$, the (informal) content of the Sudakov theorem is the existence of a deterministic optimal transport map from an absolutely continuous source measure $\mu \in \mathcal{P}_{\mathrm{ac}}(\mathbb{R}^n)$ to an arbitrary target $\nu \in \mathcal{P}(\mathbb{R}^n)$. A common strategy is to reduce the $n$-dimensional problem to a family of one-dimensional problems along (essentially) irreducible transport rays, and then glue the resulting solutions to obtain a global transport map. However, Sudakov's original argument \cite{Sudakov1979GeometricPI} contains a gap: after disintegrating $\mu$ along transport rays, the resulting one-dimensional conditional measures need not be absolutely continuous with respect to $\mathcal{H}^1$. Ambrosio and Pratelli \cite{Ambrosio2003} pointed out this issue, provided a counterexample, and established existence via a stability argument based on $\Gamma$-convergence of the Kantorovich functional.

On Riemannian manifolds, Feldman and McCann~\cite{feldman2001monge} proved existence for compactly supported measures $\mu,\nu\in\mathcal{P}_{\mathrm{ac}}(M)$, and Figalli~\cite{Figalli2007} extended the theory to non-compact manifolds using Ma\~n\'e potentials for supercritical Lagrangians. It is also worth mentioning that Bianchini and Cavalletti generalized the Monge problem to geodesic spaces~\cite{BianchiniCavalletti2013}. In particular, using the $(K,N)$-measure contraction property, they obtained a derivation similar to that in this paper for the absolutely continuous pushforward map (see Lemma~\ref{lemma: Absolutely continuous pushforward}). Furthermore, dimensional reduction arguments have been used to prove various geometric inequalities on metric measure spaces; see~\cite{cavalletti2018overviewl1optimaltransportation,klartag2014needledecompositionsriemanniangeometry}.

\subsection{Monge problem}
Let $\mu,\nu\in\mathcal{P}(M)$ be two probability measures. For a Borel map $T:M\to M$, its pushforward acts on $\mu$ by
\[
(T_{\#}\mu)(A)=\mu\bigl(T^{-1}(A)\bigr)\qquad\text{for every Borel set }A\subseteq M.
\]
Given a cost function $c:M\times M\to\mathbb{R}$, the Monge problem is
\[
\mathbf{M}_c(\mu,\nu)
=\inf\left\{\int_M c\bigl(x,T(x)\bigr)\,d\mu(x):\ T_{\#}\mu=\nu\right\}.
\]
We seek not only the value of this infimum but also an optimal map $T$ attaining it.

Kantorovich introduced a relaxation of the Monge problem  by allowing general transport plans (couplings) between $\mu$ and $\nu$. These couplings are elements of
\[
\Pi(\mu,\nu)=\left\{\pi\in\mathcal{P}(M\times M):\ p_{\#}\pi=\mu,\ q_{\#}\pi=\nu\right\},
\]
and are used to define the Kantorovich functional
\[
\mathbf{K}_c[\pi]=\int_{M\times M} c(x,y)\,d\pi(x,y).
\]
The Kantorovich problem is to minimize $\mathbf{K}_c[\pi]$ over $\pi\in\Pi(\mu,\nu)$. If $T$ is a transport map and $\pi=(\mathrm{Id}\times T)_{\#}\mu$, then
\[
\int_M c\bigl(x,T(x)\bigr)\,d\mu(x)=\int_{M\times M} c(x,y)\,d\pi(x,y),
\]
so Monge's formulation is contained in Kantorovich's one. An important advantage is that, under mild assumptions, minimizers for the Kantorovich problem exist; one then studies when an optimal plan is induced by a map.

Throughout this paper we take $c(x,y)=d(x,y)$, where $d$ is the Riemannian distance,
\[
d(p,q)=\inf\left\{\int_0^T\sqrt{g_{\gamma(t)}\bigl(\dot\gamma(t),\dot\gamma(t)\bigr)}\,dt:\ \gamma\text{ is piecewise smooth, }\gamma(0)=p,\gamma(T)=q\right\}.
\]
In this case, the dual Kantorovich problem is
\[
\sup\left\{\int_M u\,d(\mu-\nu):\ u\in\mathrm{Lip}_1(M,d)\right\},
\]
where
\[
\mathrm{Lip}_1(M,d)=\left\{u:M\to\mathbb{R}:\ |u(x)-u(y)|\le d(x,y)\ \forall x,y\in M\right\}.
\]
Such 1-Lipschitz functions are called Kantorovich potentials.

\subsection{Main Theorem}
Let $u \in \operatorname{Lip}_1(M, d)$ be a Kantorovich potential. Its subdifferential $\partial u$ induces a family of transport rays. On a Riemannian manifold, these rays are geodesics oriented in the direction of mass transport. We denote the transport set by
\[
\mathcal{T}_{\mathrm{e}}:=\bigcup_{y \in \mathcal{S}} \mathcal{R}(y) \subset M,
\]
where $\mathcal{S}$ is a suitable transversal set (a set of representatives for the rays). Along these rays, there exists a Borel density function $D(t, y)$ describing the expansion or contraction (the Jacobian factor); this represents the area occupied by a unit mass during transport.
\begin{theorem}\label{thm: explicit disintegration of tr}
\emph{(Explicit disintegration of the transport set)}
For every $\varphi\in L^1(M,\mathrm{Vol}_g)$,
\[
\int_{\mathcal{T}_{\mathfrak{e}}} \varphi(x)\,d\operatorname{Vol}_g(x)
=\int_{\mathcal{S}}\left(\int^{u(a(y))}_{u(b(y))} \varphi\bigl(\exp_y[(t-u(y))\nabla u(y)]\bigr)\,D(t,y)\,dt\right) d\mathcal{H}^{n-1}(y),
\]
where $a(y)$ and $b(y)$ are the endpoints of the maximal transport ray through $y$.
\end{theorem}

This is a direct generalization of Caravenna's result that deals with ambient Euclidean space.

\section*{Acknowledgements}
I am deeply grateful to Prof.\ Wilhelm Klingenberg for his invaluable guidance and insightful suggestions throughout this work. I would also like to extend my sincere thanks to Prof.\ Luigi Ambrosio and Prof.\ Laura Caravenna for their thoughtful feedback and advice.
\section*{Statements and Declarations}
The author has no relevant financial or non-financial interests to disclose. No data was collected in the course of this research.

\section{Disintegration of transport sets}
In this section we introduce the transport sets following Caravenna \cite[p. 377]{caravenna2011a} and describe a decomposition into cylindrical sets that will be convenient later. We also recall some results of Feldman and McCann \cite{feldman2001monge}. Under the standing assumption $\operatorname{Sec}(M)\ge K>0$, the manifold $(M,g)$ is necessarily compact, by Bonnet--Myers theorem \cite{jost2017riemannian}.
\subsection{Elementary structure of the transport set}
We begin by studying the Borel measurability of the transport sets $\mathcal{T}$ and $\mathcal{T}_{\mathfrak{e}}$. We will also establish the Borel measurability of several multivalued functions.
\begin{defn}
Let $u$ be a Kantorovich potential on $M$. The outgoing set from $x\in M$ is
\[
\mathcal{P}(x):=\left\{z\in M:\ u(z)=u(x)-d(x,z)\right\},
\]
and the incoming set to $x\in M$ is
\[
\mathcal{P}^{-1}(x):=\left\{z\in M:\ u(z)=u(x)+d(x,z)\right\}.
\]
We define the (transport) ray through $x$ by
\[
\mathcal{R}(x):=\mathcal{P}(x)\cup\mathcal{P}^{-1}(x) \subset M.
\]
Along transport rays, $u$ decreases at the maximal rate. Moreover, on the transport set $\mathcal{T}$ the vector field $\nabla u$ points in the direction opposite to the transporting geodesics $c:[0,1]\to M$ (whenever $\dot c(t)\neq 0$). In particular,
\[
\nabla u(z_0)=-\frac{\dot c(t_0)}{\|\dot c(t_0)\|_{z_0}}.
\]
\end{defn}

\begin{defn}
Let $u$ be a Kantorovich potential on $M$. Define the set of optimal pairs
$$
\partial u:=\{(x, y) \in M \times M: u(x)-u(y)=d(x, y)\} .
$$
which is the sub-differential of a 1-Lipschitz potential. For any $(x, y) \in \partial u$, we denote by $\rrbracket x, y \llbracket$ an open geodesic segment connecting $x$ and $y$ (the relative interior of a geodesic with endpoints $x$ and $y$), and by $\llbracket x, y \rrbracket$ the corresponding closed geodesic segment. The transport set $\mathcal{T}$ is the union of all such open geodesic segments in $M$:
$$
\mathcal{T}:=\bigcup_{(x, y) \in \partial u} \rrbracket x, y \llbracket \;  \subset M.
$$
The transport set including endpoints is
$$
\mathcal{T}_{\mathfrak{e}}:=\bigcup_{(x, y) \in \partial u, x \neq y} \llbracket x, y \rrbracket \subset M.
$$
We also introduce the (possibly multivalued) endpoint maps along rays: for $y$ on a ray, let $a(y)$ denote the upper endpoints (where $u$ attains its maximum on that ray) and let $b(y)$ denote the lower endpoints (where $u$ attains its minimum on that ray).
\end{defn}

\begin{lemma}
(Lemma 9, \cite{feldman2001monge}) Let $\mathcal{R}_1$ and $\mathcal{R}_2$ be transport rays with $\mathcal{R}_1\neq\mathcal{R}_2$ and $\mathcal{R}_1\cap\mathcal{R}_2\neq\emptyset$. Then either
\begin{itemize}
    \item $\mathcal{R}_1 \cap \mathcal{R}_2 = \{p\}$ and $p$ is an endpoint (upper or lower), or
    \item $\mathcal{R}_1 \cap \mathcal{R}_2 = \{p,q\}$ where $p$ is an upper endpoint and $q$ is a lower endpoint.
\end{itemize}
In particular, an interior point of a transport ray cannot lie on any other transport ray.
\end{lemma}

\begin{remark}
On a Riemannian manifold, the cut locus $\mathrm{Cut}(a)$ of a point $a\in M$ is the set of points $x\in M$ that can be joined to $a$ by more than one minimizing geodesic. In particular, if $a\in\mathcal{T}$, then any point of $\mathrm{Cut}(a)$ lies either in $\mathcal{T}_{\mathfrak{e}}\setminus\mathcal{T}$ or in $M\setminus\mathcal{T}_{\mathfrak{e}}$.
\end{remark}

Let us recall a basic property of multivalued functions. Let $X$ and $Y$ be metric spaces, and let $F: X \rightrightarrows Y$ be a multivalued map. If $\operatorname{graph}(F)$ is $\sigma$-compact, then
\begin{itemize}
    \item For any closed subset $C \subseteq Y$, the preimage $F^{-1}(C)$ is $\sigma$-compact in $X$.
    \item For any closed subset $D \subseteq X$, the image $F(D)$ is $\sigma$-compact in $Y$.
\end{itemize}
Let us denote $X = \mathrm{supp}(\mu)$ and $Y = \mathrm{supp}(\nu)$.
\begin{lemma}
The multivalued functions $\mathcal{P}, \mathcal{P}^{-1}, \mathcal{R}$ have a $\sigma$-compact graph. Therefore the transport sets $\mathcal{T}$ and $\mathcal{T}_{\mathfrak{e}}$ are $\sigma$-compact and Borel measurable.
\end{lemma}
\begin{proof}
We follow the argument of Caravenna \cite{caravenna2011a}.

\emph{Step 1: $\operatorname{graph}(\mathcal{P})$ is closed.} Let $(x_k,z_k)\to(x,z)$ in $M\times M$ with $z_k\in\mathcal{P}(x_k)$. By definition,
\[
u(z_k)=u(x_k)-d(x_k,z_k).
\]
Passing to the limit gives $u(z)=u(x)-d(x,z)$, i.e. $z\in\mathcal{P}(x)$. Hence $\operatorname{graph}(\mathcal{P})$ is closed.

Since $M$ is compact, $M\times M$ is compact; therefore $\operatorname{graph}(\mathcal{P})$ is compact, hence $\sigma$-compact. The same argument applies to $\operatorname{graph}(\mathcal{P}^{-1})$, and since
\[
\operatorname{graph}(\mathcal{R})=\operatorname{graph}(\mathcal{P})\cup\operatorname{graph}(\mathcal{P}^{-1}),
\]
we also obtain that $\operatorname{graph}(\mathcal{R})$ is $\sigma$-compact.

\emph{Step 2: remove the diagonal.} Note that
\[
\operatorname{graph}(\mathcal{P}\setminus\mathrm{Id})
=\operatorname{graph}(\mathcal{P})\setminus\operatorname{graph}(\mathrm{Id}).
\]
Since $\operatorname{graph}(\mathrm{Id})$ is closed in $M\times M$, its intersection with $\operatorname{graph}(\mathcal{P})$ is closed in $\operatorname{graph}(\mathcal{P})$, hence the complement $\operatorname{graph}(\mathcal{P}\setminus\mathrm{Id})$ is open in $\operatorname{graph}(\mathcal{P})$ and therefore $\sigma$-compact. The same holds for $\operatorname{graph}(\mathcal{P}^{-1}\setminus\mathrm{Id})$.

\emph{Step 3: conclusion for $\mathcal{T}$ and $\mathcal{T}_{\mathfrak{e}}$.} Using the characterization
\[
\mathcal{T}=(\mathcal{P} \backslash \operatorname{Id})(X) \cap\left(\mathcal{P}^{-1} \backslash \operatorname{Id}\right)(Y), \quad \mathcal{T}_{\mathfrak{e}}=(\mathcal{P} \backslash \operatorname{Id})(X) \cup\left(\mathcal{P}^{-1} \backslash \operatorname{Id}\right)(Y),
\]
we deduce that $\mathcal{T}$ and $\mathcal{T}_{\mathfrak{e}}$ are $\sigma$-compact. In particular, they are Borel measurable.
\end{proof}
\begin{lemma}
The endpoint maps $a,b:\mathcal{T}_{\mathfrak{e}}\rightrightarrows\mathcal{T}_{\mathfrak{e}}$ are Borel measurable.
\end{lemma}
\begin{proof}
For each $z\in\mathcal{T}_{\mathfrak{e}}$, Feldman and McCann define in \cite{feldman2001monge}
\[
\begin{aligned}
\alpha(z)&:=\sup\left\{d(y,z):\ y\in Y,\ u(z)-u(y)=d(y,z)\right\},\\
\beta(z)&:=\sup\left\{d(x,z):\ x\in X,\ u(x)-u(z)=d(x,z)\right\}.
\end{aligned}
\]
Then $\alpha,\beta:M\to\mathbb{R}_{\ge 0}\cup\{-\infty\}$ are upper semicontinuous. Geometrically, $\alpha(z)$ and $\beta(z)$ represent the distances from $z$ to the lower and upper endpoints, respectively, of the corresponding transport ray $\mathcal{R}_z$.

Consequently, the endpoint maps can be written as
\[
\begin{aligned}
 a(z)&=\left\{y\in\mathcal{T}_{\mathfrak{e}}:\ \beta(y)=0,\ u(z)=u(y)-d(z,y)\right\},\\
 b(z)&=\left\{y\in\mathcal{T}_{\mathfrak{e}}:\ \alpha(y)=0,\ u(y)=u(z)-d(z,y)\right\}.
\end{aligned}
\]
To prove that $a$ and $b$ are Borel measurable, it suffices to show that their graphs are Borel subsets of $\mathcal{T}_{\mathfrak{e}}\times\mathcal{T}_{\mathfrak{e}}$. We treat $\operatorname{graph}(a)$; the case of $b$ is analogous. Observe that
\[
\operatorname{graph}(a)=\left(\mathcal{T}_{\mathfrak{e}}\times\beta^{-1}(0)\right)\cap\left\{(z,y):\ u(z)=u(y)-d(z,y)\right\}
=\left(\mathcal{T}_{\mathfrak{e}}\times\beta^{-1}(0)\right)\cap\operatorname{graph}(\mathcal{P}^{-1})
\subset M \times  M,\]
which is a Borel subset.
\end{proof}
\begin{theorem}
The transport set $\mathcal{T}$ satisfies the following properties:
\begin{enumerate}
\item The vector field $\nabla u$ has the countable Lipschitz property, i.e. there exists a $\mu$-negligible set $N\subset\mathcal{T}$ and an increasing sequence of compact sets $(K_m)_{m\in\mathbb{N}}$ such that $\nabla u\vert_{K_m}$ is Lipschitz for every $m$ and
\[
\bigcup_{m\in\mathbb{N}} K_m = \mathcal{T}\setminus N.
\]
\item $\operatorname{Vol}_g\bigl(\mathcal{T}_{\mathfrak{e}}\setminus\mathcal{T}\bigr)=0$.
\end{enumerate}
\end{theorem}
\begin{proof}
Part~(2) is proved in Lemma~24 of \cite{feldman2001monge}. Part~(1) also follows from the arguments in \cite{feldman2001monge}; for instance, it may be derived using Lemma~19 therein.
\end{proof}

\subsection{Sheaf sets and cylinder sets}
In this subsection we draw ideas from \cite{caravenna2011a} and \cite{feldman2001monge} to define sheaf sets $\{\mathcal{Z}_{\mathfrak{e}}\}$ and cylinder sets $\{\mathcal{K}\}$ via a modification of Feldman--McCann's ray clusters. These sets form a countable cover of $\mathcal{T}_{\mathfrak{e}}$, and we will derive a Fubini--Tonelli decomposition of measures on them. This lays the foundation for the explicit disintegration formula in our Theorem 1.
\begin{lemma}\label{lemma : bi-lipschitz parametrization of levelsets}
(Lemma 17, \cite{feldman2001monge}). Let $u: M \rightarrow \mathbb{R}^1$ be a Lipschitz function, $v \in \mathbb{R}^1$, and $S_v$ the level set $\{x \in M \mid u(x)=v\}$. Then the set
$$
S_v \cap\{x \in M \mid u \text { is differentiable at } x \text { and } \nabla u(x) \neq 0\}
$$
has a countable covering consisting of Borel sets $S_v^i \subset S_v$, such that for each $i \in \mathbf{N}$ there exist Lipschitz maps $U: S_v^i \rightarrow \mathbb{R}^{n-1}$ and $V: \mathbb{R}^{n-1} \rightarrow M$ (i.e., Lipschitz maps between metric spaces $(M, d)$ and $\left(\mathbb{R}^{n-1},|\cdot|\right)$, where $|\cdot|$ is a Euclidean metric on $\mathbb{R}^{n-1}$), satisfying
$$
V(U(z))=z \quad \text { for all } z \in S_v^i .
$$
Also, the $U\left(S_v^i\right)$ are Borel subsets of $\mathbb{R}^{n-1}$.
\end{lemma}

An immediate consequence of this bi-Lipschitz parametrization is that
\[
E:=S_v \cap \left\{x \in M : u \text{ is differentiable at } x \text{ and } \nabla u(x) \neq 0\right\}
\]
is $\mathcal{H}^{n-1}$-countably rectifiable.

\begin{theorem}\label{thm: approximate tangent space}
(Theorem 3.1.6, \cite{Simon1983GMT}) Suppose $E  \subset \mathbb{R}^{n+k}$ is $\mathcal{H}^n$-measurable with $\mathcal{H}^n(E \cap K)<\infty$ for each compact $K \subset \mathbb{R}^{n+k}$. Then $E$ is countably $n$-rectifiable if and only if the approximate tangent space $T_x E$ exists for $\mathcal{H}^n$-a.e. $x \in E$. Suppose $E  \subset \mathbb{R}^{n+k}$ is $\mathcal{H}^n$-measurable with $\mathcal{H}^n(E \cap K)<\infty$ for each compact $K \subset \mathbb{R}^{n+k}$. Then $E$ is countably $n$-rectifiable if and only if the approximate tangent space $T_x E$ exists for $\mathcal{H}^n$-a.e. $x \in E$.
\end{theorem}

In fact, this result also holds on Riemannian manifolds: rectifiability and the existence of approximate tangent spaces are local properties, and a Riemannian manifold is locally modeled on Euclidean space via smooth (hence bi-Lipschitz) coordinate charts.

It is also noteworthy that our sheaf set is built upon the level sets of the Kantorovich potential $u$. This construction aligns it more closely with the ray clusters, which are the model sets introduced by Feldman and McCann.
\begin{defn}
(Lemma 21, \cite{feldman2001monge}). Fix $v \in \mathbb{R}^1$, a Kantorovich potential $u$ on $M$, and a Borel cover $\left\{S_v^i\right\}_i$ of the level set $S_v:=\{x \in M \mid u(x)=v\}$. Let $i \in \mathbf{N}$ and let $B$ be a Borel subset of $S_v^i$. For each $j \in \mathbf{N}$. Let the cluster $(\mathcal{T}_{\mathfrak{e}})_{v i j}(B):=\bigcup \mathcal{R}_z$ denote the union of all transport rays $\mathcal{R}_z$ which intersect $B$, and for which the point of intersection $z \in B$ is separated from both endpoints of the ray by a distance greater than $1 / j$ :
$$
(\mathcal{T}_{\mathfrak{e}})_{v i j}(B):=\bigcup_{z \in B}\left\{\mathcal{R}(z): d(a(z), z)>\frac{1}{j}, d(b(z), z)>\frac{1}{j}\right\} \subset M .
$$
The same cluster, but with ray ends omitted, is denoted by $\mathcal{T}_{v i j}(B):=\bigcup_z\left(\mathcal{R}_z\right) \backslash \bigcup_z a(z) \cup b(z)$.
\end{defn}

We now modify the model ray clusters to build sheaf sets. Fix $v\in\mathbb{R}$ and choose a countable Borel covering
\[
\bigcup_{i=1}^{\infty} S_{v i}=S_v\cap F,
\]
where $F=\{y\in M: u\text{ is differentiable at }y\text{ and }\nabla u(y)\neq 0\}$. By Lusin's theorem, we can find compact sets $\{K_{vij}\}_j\subset S_{pi}$ such that $\mathcal{H}^{n-1}\bigl((S_{vi}\cap F)\setminus\bigcup_j K_{vij}\bigr)=0$ and the restrictions $a\vert_{K_{vij}}$, $b\vert_{K_{vij}}$, and $(\nabla u)\vert_{K_{vij}}$ are continuous.

Since we work on a Riemannian manifold and will later use the exponential map, we further decompose each $K_{vij}$ into compact pieces $K_{vij}^\ell$ with
\[
\operatorname{diam}(K_{vij}^\ell)\le \delta/2,
\]
where $\delta>0$ is chosen so that $\delta<\inf\{\operatorname{inj}(\mathcal{T}_{\mathfrak{e}})\}$. For simplicity, re-index these compact sets as $\{K_{vq}\}_{q\in\mathbb{N}}$ and let $\hat K_{vq}$ denote their disjoint counterparts.
For $k\in\mathbb{N}$ and $l,m\in\mathbb{Z}$ with $l<m$, define the sheaf sets
\[
\mathcal{Z}_{qklm}^v
:=\left\{y:\ v+\frac{l-1}{2^k}\le u(y)\le v+\frac{m+1}{2^k}\right\}\cap\left\{y:\ |u(y)-v|\le \frac{\delta}{2}\right\}\cap\bigcup_{z\in K_{vq}} \llbracket a(z),b(z)\rrbracket.
\]
These sheaf sets consist of portions of rays confined between suitable level sets of $u$ and provide a cover of $\mathcal{T}$. Equivalently,
\[
\mathcal{Z}_{qklm}^v
=\left\{y:\ v+\max\left\{\frac{l-1}{2^k},-\frac{\delta}{2}\right\}\le u(y)\le v+\min\left\{\frac{m+1}{2^k},\frac{\delta}{2}\right\}\right\}\cap\bigcup_{z\in K_{vq}} \llbracket a(z),b(z)\rrbracket.
\]
\begin{lemma}
The family $\left\{\mathcal{Z}^v_{qklm}\right\}$ can be refined into a partition of the transport set $\mathcal{T}$.
\end{lemma}
\begin{proof}
Define a set of discrete points:
\[
\Gamma_k=\left\{\frac{m}{2^k}: m \in \mathbb{Z}\right\}.
\]
This set partitions the range of $u$ into disjoint half-open intervals $I_m=\left[m / 2^k,(m+1) / 2^k\right)$. Choose $k$ large enough so that the injectivity constraint is automatically satisfied, since $|u(y)- m/2^k|\leq 1/2^k < \delta/2$. The disjoint sheaf sets are then indexed by
\[
\mathcal{Z}'_{m,q} = \left\{y: u(y) \in\left[\frac{m}{2^k}, \frac{m+1}{2^k}\right)\right\} \cap \bigcup_{z \in \hat{K}_{mq}}\llbracket a(z), b(z)\rrbracket.
\]
We claim that the family $\{\mathcal{Z}'_{m,q}\}$ covers almost all of $\mathcal{T}$. Let $y \in \mathcal{T}$ be arbitrary. There exists a unique integer $m$ such that $u(y) \in\left[m / 2^k,(m+1) / 2^k\right)$. Let $\mathcal{R}(y)$ be the transport ray passing through $y$. Then $\mathcal{R}(y) \cap \hat{K}_{mq}=\{z\}$ for a unique compact base $\hat{K}_{mq}$, and the point $z$ lies at positive distance from each end of the ray.

In fact, if we use the closed intervals $\left[m / 2^k,(m+1) / 2^k\right]$ instead, we still cover almost all of $\mathcal{T}$, since
\[
\mathcal{H}^n\left(\mathcal{Z}'_{m, q} \cap \mathcal{Z}'_{m+1, q^{\prime}}\right)=0.
\]
\end{proof}

We now turn to a reparametrization of the sheaf sets, a definition that is easier to work with. In the remainder of this paper, our focus will be on the cylinder set.
\begin{defn}\label{def: cylinder set}
(Cylinder Set).
Consider the flow function moving points along the transport rays,
\[
\Phi: Z \times \left[h^{-}, h^{+}\right] \to M, \qquad \Phi(p,t)=\exp_p\bigl(t\nabla u(p)\bigr).
\]
Here $Z:=\hat{K}_{u_0q}$ is a compact base set, consisting of points where the transport rays intersect the level set of $u$.
We define the cylinder set subordinated to the compact base $Z$ by
\[
\mathcal{K}:=\Phi\bigl(Z\times[h^-,h^+]\bigr)\cap\bigcup_{z\in Z}\llbracket a(z), b(z)\rrbracket,
\]
where $Z\subset\{x\in M: u(x)=u_0\}$ is compact and $h^-<0<h^+$.
A cylinder is simply a reparametrization of a sheaf set, and therefore forms a partition of the transport set.
\end{defn}

\begin{figure}[t]
\centering
\begin{tikzpicture}[x=1cm,y=1cm,scale=0.95]

  \path[fill=gray!25,draw=gray!70,line width=0.6pt]
    (4.72,0.95) .. controls (4.45,1.70) and (4.45,2.55) .. (4.72,3.25)
    .. controls (4.78,3.42) and (4.90,3.53) .. (5.05,3.58)
    .. controls (5.25,3.64) and (5.43,3.50) .. (5.50,3.25)
    .. controls (5.78,2.50) and (5.78,1.68) .. (5.50,0.95)
    .. controls (5.43,0.75) and (5.25,0.63) .. (5.05,0.66)
    .. controls (4.90,0.68) and (4.78,0.79) .. (4.72,0.95) -- cycle;
  \draw[dashed,gray!70,line width=0.7pt]
    (4.80,1.05) .. controls (4.55,1.75) and (4.55,2.45) .. (4.80,3.15);
  \draw[dashed,gray!70,line width=0.7pt]
    (5.42,1.05) .. controls (5.67,1.75) and (5.67,2.45) .. (5.42,3.15);

  \path[fill=blue!35,draw=blue!60!black,dashed,line width=0.6pt]
    (4.86,1.35) .. controls (5.05,1.15) and (5.34,1.18) .. (5.45,1.42)
    .. controls (5.60,1.78) and (5.58,2.55) .. (5.38,2.88)
    .. controls (5.22,3.12) and (4.95,3.12) .. (4.82,2.88)
    .. controls (4.62,2.55) and (4.62,1.78) .. (4.78,1.42)
    .. controls (4.81,1.37) and (4.83,1.36) .. (4.86,1.35) -- cycle;
  \node at (5.10,2.15) {$Z$};

  \foreach \yy/\a in {1.35/0.26,1.75/0.22,2.15/0.18,2.55/0.14,2.95/0.10} {
    \draw[red!75!black,line width=0.65pt]
      (-1.1,\yy)
      .. controls (1.2,\yy+0.55+2*\a) and (3.2,\yy-0.45-2*\a) .. (4.95,\yy)
      .. controls (6.6,\yy+0.50+2*\a) and (8.2,\yy-0.40-2*\a) .. (9.6,\yy+0.08);
  }

  \node[anchor=east] at (-1.15,1.20) {$\mathcal{Z}$};
\end{tikzpicture}
\caption{schematic drawing of a sheaf set: the transport rays are mutually disjoint geodesics (red). They pierce the hypersurface slice (gray). The set $Z$ is a subset of a level set of $u$, and $Z$ is a compact base set.}
\label{fig:riemannian-geodesics-slice}
\end{figure}
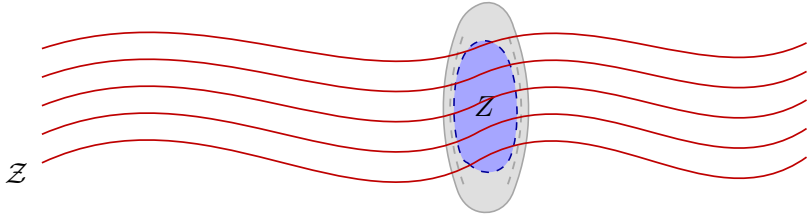
\subsection{Approximation using Jacobi fields}
Recall Proposition~15 in \cite{feldman2001monge}: $\nabla u$ is Lipschitz on each level set of $u$; hence, for each fixed $t$, the map
\[
\Phi^t: Z\to M,\qquad \Phi^t(p):=\Phi(p,t),
\]
is Lipschitz.
Our goal is to estimate the expansion factor of the compact set $Z$ under the flow $\Phi^t$ along the cylinder. By Rademacher's theorem, this factor, interpreted as the Jacobian determinant of $
\Phi^t$, is given by:
\[
\mathcal{H}^{n-1}\bigl(\Phi^t(Z)\bigr) = \int_Z J_{n-1}(\Phi^t)(p) \, d\mathcal{H}^{n-1}(p).
\]
To approximate this quantity, we adapt a method from Caravenna's construction \cite{caravenna2011a} to the Riemannian setting. We will use the approximate tangent space to $Z$ (since it is $(n-1)$-rectifiable) together with Jacobi fields in the ambient manifold.
Let us start with a sheaf set (cylinder) $\mathcal{Z}_e(Z)$ with compact basis $Z$ as follows:
$$\mathcal{Z}_e = \bigcup_{y \in Z} \llbracket a(y) ,b(y)\rrbracket,$$
where
$$u(y) - u(a(y)) = -d(y,a(y)) < h^- < 0,\quad u(y) -u(b(y)) = d(y,b(y)) > h^+ > 0.$$
Approximate the potential $u$ with the sequence of potentials $\{u_I\}_{I \in \mathbb{N}}$ by choosing a dense sequence $\{a_i\}_i \subset \Phi_{h^-}(Z)$:
$$u_I(x) = \max \left\{u(a_i) -d(x,a_i): i= 1,\dots,I\right\}.$$
Since $u: M \to \mathbb{R}$ is $1$-Lipschitz, it is uniformly continuous on $\Phi_{h^-}(Z)$. Notice that $\{u_I(x)\}$ is monotonically increasing and $u_I(x) \leq u(x)$. We first show that $u_I(x)$ converge pointwise to $u(x)$ and we then show this convergence is uniform. Since $\{a_k\}_{k \in \mathbb{N}}$ is a dense set, choose $x \in \overline{\mathcal{Z}_e \cap \{\Phi_t(Z) : t \geq h^-\}}$ such that $d(a_k, x) < \varepsilon/2$.
Notice that we have
$$u(a_k) \geq u(x) - d(x,a_k).$$
Then
$$u(a_k) - d(x,a_k) \geq u(x) - 2d(x,a_k) > u(x) - \varepsilon.$$
Choose $I \geq k$ and taking maximum over all $j = \{1,\dots, I\}$ of both sides:
$$u_I(x) > u(x) - \varepsilon.$$
Therefore $u_I$ converge pointwise to $u$ on the closure of $\mathcal{Z}_e \cap \{\Phi_t(Z) : t \geq h^-\}$. By Dini's theorem, $u_I$ is uniform convergent to $u$ on the closure of $\mathcal{Z}_e \cap\left\{\Phi_t(Z): t \geq h^{-}\right\}$.

Consider now the vector fields of ray directions
$$d_I(x) = \sum_{i=1}^I d^i(x) \chi_{\Omega^I_i}(x) ,\quad d^i(x) = -\nabla_x d(a_i,x).$$
where the open sets $\Omega^I_i$ are
$$
\begin{aligned}
\Omega_i^I & =\left\{x \in M: u\left(a_i\right)-d(x,a_i)>u\left(a_j\right)-d(x,a_j), j \in\{1 \ldots I\} \backslash i\right\} \\
& =\text { interior of }\left\{x: u\left(a_i\right)=u_I(x)+d(x,a_i)\right\},
\end{aligned}
$$
and the $(n-1)$-$\mathcal{H}^{n-1}$ rectifiable boundary can be written as
$$
\begin{aligned}
J_I & =\bigcup_{i \neq j}\left(\bar{\Omega}_i \cap \bar{\Omega}_j\right) \\
& =\left\{x: \exists i, j, i \neq j, u\left(a_i\right)-d(a_i,x)=u\left(a_j\right)-d(a_j,x)\right\}.
\end{aligned}
$$
Notice that when $I \to \infty$, $\Omega^I_i$ as a set converges $\mathcal{L}^n$-a.e. to $P(a_i)$. Indeed, 
\begin{align*}
&\Omega^I_i =\left\{x \in M : u(a_i)- d(x,a_i) = \max_{1 \leq j \leq I}\left\{u(a_j) - d(x,a_j)\right\}\right\}\\
&\stackrel{I\to \infty}{\longrightarrow} \left\{x \in M : u(a_i) - d(x,a_i) = u(x)\right\} = \mathcal{P}(a_i),
\end{align*}
which is a transport ray starting from $a_i$. 

Since the closure of $\mathcal{Z}_e \cap \{\Phi_t(Z) : t \geq h^-\}$ is compact, let $\{(d_I)_j\}_{j \in \mathbb{N}}$ be a selection of subsequences with corresponding convergent subsequences of $\{a_{i_j}\}_{j \in \mathbb{N}}$. Hence 
$$\lim_{j \to \infty} (d_I)_j (x) = -\lim_{j \to \infty}\nabla d(a_{i_j},x) = \bar{d} = -\nabla d(a,x), \text{ with } a \in \Phi_{h^-}(Z). $$
Notice that $\{u_{I_j}\}$ defines a sequence of potentials and
$$
u_{I_j}\left(\left(a_i\right)_j\right)=u_{I_j}(x)+d\left(x,\left(a_i\right)_j\right)
$$
for all $x$ on the ray defined by the approximating potential.
By taking the limit of both sides, we get $u(a) = u(x) + d(x,a)$.

Suppose two different subsequences
converge to two limits:
$$
a_{i_j} \rightarrow a \text { and } a_{i_k^{\prime}} \rightarrow a^{\prime}.
$$
From the definition of $u_{I_j}$, 
$$u(a)- d(x,a) = u(a') - d(x,a').$$
But this is only true if $a$ and $a'$ are on the same transport ray (geodesic). Since $a,a' \in \Phi_{h^-}(Z)$, it must be that $a =a'$. All convergent subsequences converge to the same limit, and therefore $d_I$ converges $\mathcal{L}^n$-a.e. to $d = \bar{d}$ on this $d$-cylinder.

Define the map $\Phi^{-t}_I$ similarly to $\Phi_t$ moves points along the rays relative to $d_I$. Within $\Omega^I_i$ the map $\Phi^{-t}_I$ moves points toward $a_i$ for $i \leq I$. As a consequence, for $S \subset \Omega^I_i \cap [\Phi_{h}(Z)]$ and $h-h^- > t \geq 0$, we can compute a contraction factor. Indeed, since there exists a bi-Lipschitz homeomorphism from $S$ to its image in $\mathbb{R}^{n-1}$, $S$ is $\mathcal{H}^{n-1}$-rectifiable. This means that for $\mathcal{H}^{n-1}$-almost every point $y \in S$, there exists an approximate tangent space $V \subset T_y M$, which is an ($n-1$)-dimensional linear subspace:
$$V \subset \mathbb{R}\dot{\gamma}(0) \oplus \langle \nabla_{\dot{\gamma}(0)}V_2,\dots \nabla_{\dot{\gamma}(0)}V_n\rangle,$$
where $\gamma : [0,r] \to M$ is the radial geodesic connecting $y$ to $a_i$. We can write
$$\Phi^{-t}_I(y) = \exp_{a_i}\left(\frac{r(y)-t}{r(y)}\exp^{-1}_{a_i}(y)\right),\quad r(y) = d(a_i, y).$$
We will give a lower bound for the Jacobian:
$$
\operatorname{Jac}\left(\Phi^{-t}_I|_S\right)(y):=\sqrt{\operatorname{det}\left(\left(\left.L\right|_V\right)^*\left(\left.L\right|_V\right)\right)}
$$
where $L=\left(d \Phi_I^{-t}\right)_y: T_y M \rightarrow T_{\Phi_I^{-t}(y)} M$ is the linear differential of the ambient map. Let us start with a 1-parameter family of geodesics
$$\Gamma(\tau,s) = \gamma_\tau(s) = \exp_{a_i}\left(s \frac{\exp^{-1}_{a_i}(c(\tau))}{d(a_i,c(\tau))}\right),$$
where $c: (-\varepsilon,\varepsilon) \to M$ is a curve such with initial condition $c(0) = y$ and $c'(0) = v \in \operatorname{Tan}(y,S)$. We define $r(\tau) = d(a_i,c(\tau))$ with $r(0)= d(a_i,y) = r_y$. The Jacobi field \cite{jost2017riemannian} along $\gamma(s)$ is defined as
$$J(s) = \frac{\partial\Gamma}{\partial\tau}(0,s).$$
Although the endpoints of the geodesics depend on $\tau$, we can extend the domain of $\Gamma$ to $(-\varepsilon, \varepsilon) \times[0, R]$ for sufficiently large $R$.
We have $J(0) = 0$ and $\Gamma(\tau,r(\tau)) = c(\tau)$. Differentiate this condition with respect to $\tau$ at $0$:
$$\frac{\partial \Gamma}{\partial s}(0,r(0)) \cdot \frac{\partial r}{\partial \tau}(0) + J(r(0)) = c'(0),$$
which is
$$\gamma'(r(0)) \langle v, \nabla_y d(a_i, y)\rangle + J(r(0)) = v.$$
Hence
$$J(r_y) = v - \langle v, \nabla_y d(a_i,y)\rangle \dot{\gamma}(r_y)= v^\perp.$$
Therefore at $s = r_y$, the Jacobi field is the component of $v$ orthogonal to the ray $\gamma(s)$.
To use the Rauch comparison theorem, we also need to check that $\langle J(s), \dot{\gamma}(s) \rangle = 0$ for all $t \in [0,r_y]$. But this follows immediately from the fact that $\langle J(s), \dot{\gamma}(s) \rangle$ is linear.
The differential $d(\Phi^{-t}_I)_y$ acting on $v$ is given by:
\begin{align*}
d(\Phi^{-t}_I)_y(v) &= \frac{d}{d\tau}|_{\tau = 0}\Gamma(r(\tau)-t,c(\tau)) = \frac{\partial \Gamma}{\partial \tau}(r(0)-t,0) + \frac{\partial}{\partial s}(r(0)-t,0) \frac{\partial(r(\tau-t))}{\partial \tau}|_{\tau = 0}\\
&= J(r-t) + \dot{\gamma}(r-t) \langle v, \nabla d(a_i,y)\rangle\\
&= J(r-t) + \dot{\gamma}(r-t)\langle v,\dot{\gamma}\rangle.
\end{align*}
By orthogonality, the squared norm is
$$\left\|d(\Phi^{-t}_I)_y(v)\right\|^2 = \|J(r-t)\|^2 + \langle v, \dot{\gamma}(r) \rangle^2. $$
By the Bishop--Gromov  theorem \cite{bishop1963relation}, $s \mapsto \|J(s)\|/\mathfrak{s}_K(s)$ is a non-increasing function. Therefore,
\[
\|J(r-t)\| \geq \frac{\mathfrak{s}_K(r-t)}{\mathfrak{s}_K(r)}\|J(r)\|.
\]
With $\|v^{\perp}\|^2=\|v\|^2-\langle v, \dot{\gamma}(r)\rangle^2$, we obtain the lower bound
\[
\bigl\|d(\Phi^{-t}_I)_y(v)\bigr\|^2 \ge \left(\frac{\mathfrak{s}_K(r-t)}{\mathfrak{s}_K(r)}\right)^2\|v\|^2+\left(1-\left(\frac{\mathfrak{s}_K(r-t)}{\mathfrak{s}_K(r)}\right)^2\right)\langle v, \dot{\gamma}(r)\rangle^2.
\]
Define the unit normal vector $\mathbf{n}_S:=\nabla u$ so that $\operatorname{Tan}(y,S)\perp \mathbf{n}_S(y)$. Consider the decomposition $\operatorname{Tan}(y,S)=U\oplus W$, where
\begin{itemize}
    \item $U:=\left\{v\in\operatorname{Tan}(y,S): \langle v,\dot{\gamma}(r)\rangle=0\right\}$. With the additional constraint $\langle v,\mathbf{n}_S\rangle=0$, we have $\dim(U)=n-2$ and
    \[
    \bigl\|d(\Phi^{-t}_I)_y(v)\bigr\|^2 \ge \left(\frac{\mathfrak{s}_K(r-t)}{\mathfrak{s}_K(r)}\right)^2,\qquad \forall v\in U,\ \|v\|=1.
    \]
    \item $W:=\operatorname{Span}\{\mathbf{n}_S,\dot{\gamma}(r)\}$. Let $\theta$ be the angle between $\mathbf{n}_S$ and $\dot{\gamma}(r)$, so that
    \[
    \cos\theta=\langle \mathbf{n}_S,\dot{\gamma}(r)\rangle_y.
    \]
    Then $\langle v,\dot{\gamma}(s)\rangle=\sin\theta$ for all $v\in W$, and hence
    \begin{align*}
    \bigl\|d(\Phi^{-t}_I)_y(v)\bigr\|^2 &\ge \left(\frac{\mathfrak{s}_K(r-t)}{\mathfrak{s}_K(r)}\right)^2 + \left(1-\left(\frac{\mathfrak{s}_K(r-t)}{\mathfrak{s}_K(r)}\right)^2\right)\sin^2\theta\\
    &= \left(\frac{\mathfrak{s}_K(r-t)}{\mathfrak{s}_K(r)}\right)^2\cos^2\theta + \sin^2\theta.
    \end{align*}
\end{itemize}
Therefore, we obtain:
\begingroup
\setlength{\abovedisplayskip}{6pt}\setlength{\belowdisplayskip}{6pt}%
\setlength{\abovedisplayshortskip}{4pt}\setlength{\belowdisplayshortskip}{4pt}%
\[
\operatorname{Jac}(\Phi^{-t}_I)(y) \ge \left(\frac{\mathfrak{s}_K(r(y)-t)}{\mathfrak{s}_K(r(y))}\right)^{n-1}\sqrt{\cos^2\theta + \frac{\sin^2\theta}{\left(\frac{\mathfrak{s}_K(r(y)-t)}{\mathfrak{s}_K(r(y))}\right)^2}}.
\]
\endgroup
When analyzing a $\mathcal{K}$-cylinder, we have $\frac{\mathfrak{s}_K(r(y)-t)}{\mathfrak{s}_K(r(y))}\le 1$. Moreover, the flow $\Phi^t$ moves points along transport rays with unit speed (since $|\nabla u|=1$ on $\mathcal{T}$). Hence the distance between the two slices $\Phi^{h}(S)$ and $\Phi^{h^-}(S)$ equals $h-h^-$, which yields a uniform lower bound on the radial distance:
\begingroup
\setlength{\abovedisplayskip}{6pt}\setlength{\belowdisplayskip}{6pt}%
\setlength{\abovedisplayshortskip}{4pt}\setlength{\belowdisplayshortskip}{4pt}%
\[
 h-h^- = d\bigl(\Phi^h(S),\Phi^{h^-}(S)\bigr) = \inf_{y\in S} r(y).
\]
\endgroup
Thus a uniform lower bound is
\begingroup
\setlength{\abovedisplayskip}{6pt}\setlength{\belowdisplayskip}{6pt}%
\setlength{\abovedisplayshortskip}{4pt}\setlength{\belowdisplayshortskip}{4pt}%
\[
\operatorname{Jac}(\Phi^{-t}_I)(y) \ge \left(\frac{\mathfrak{s}_K(h-t-h^-)}{\mathfrak{s}_K(h-h^-)}\right)^{n-1}.
\]
\endgroup
Consequently,
\begingroup
\setlength{\abovedisplayskip}{6pt}\setlength{\belowdisplayskip}{6pt}%
\setlength{\abovedisplayshortskip}{4pt}\setlength{\belowdisplayshortskip}{4pt}%
\begin{equation}\label{eq: approximation by I}
\mathcal{H}^{n-1}\bigl(\Phi^{-t}_I(S)\bigr)
\ge \left(\frac{\mathfrak{s}_K(h-t-h^-)}{\mathfrak{s}_K(h-h^-)}\right)^{n-1} \mathcal{H}^{n-1}(S).
\end{equation}
\endgroup
If $K=0$, we recover (2.9) in \cite{caravenna2011a}:
\begingroup
\setlength{\abovedisplayskip}{6pt}\setlength{\belowdisplayskip}{6pt}%
\setlength{\abovedisplayshortskip}{4pt}\setlength{\belowdisplayshortskip}{4pt}%
\[
\mathcal{H}^{n-1}\bigl(\sigma_{d_I}^{-t} S\bigr)=\left(\frac{h-t-h^{-}}{h-h^{-}}\right)^{n-1} \mathcal{H}^{n-1}(S).
\]
\endgroup

\begin{remark}
Equation~\eqref{eq: approximation by I} is the point where the assumption of 
positive sectional curvature is used.
Indeed, Lemma~\ref{lemma: Absolutely continuous pushforward} requires a lower bound ($\geq$) of the form
\eqref{eq: approximation by I}, which is obtained via the comparison estimate
$\|J(r-t)\|\ge \frac{\mathfrak{s}_K(r-t)}{\mathfrak{s}_K(r)}\|J(r)\|$ along transport rays.
Moreover, the monotonicity of the factor $\frac{\mathfrak{s}_K(r-t)}{\mathfrak{s}_K(r)}$ in $r$
allows us to replace $r(y)$ by the uniform bound $\inf_{y\in S} r(y)=d\bigl(\Phi^h(S),\Phi^{h^-}(S)\bigr)$, i.e.
by the distance between two slices of the cylinder.
\end{remark}

\begin{lemma}\label{lemma: Absolutely continuous pushforward}
(Absolutely continuous pushforward) For $h^- < s \leq t < h^+$, the following inequality holds:
\begin{align*}\label{eq: absolutely continuous pushforward}
\left(\frac{\mathfrak{s}_K(h^+-t)}{\mathfrak{s}_K(h^+-s)}\right)^{n-1}
\mathcal{H}^{n-1}\left(\Phi^sS\right) 
&\leq \mathcal{H}^{n-1}\left(\Phi^tS\right)\\
&\leq
\left(\frac{\mathfrak{s}_K(t-h^-)}{\mathfrak{s}_K(s-h^-)}\right)^{n-1}
\mathcal{H}^{n-1}\left(\Phi^sS\right),
\quad \forall S \subset Z.
\end{align*}
\end{lemma}
\begin{proof}
Fix $h^- < s \leq t \leq h^+$. Consider $S \subset Z$ and first assume that $\mathcal{H}^{n-1}(\Phi^tS)> 0$. Notice that $\left\{\Omega^I_i\right\}$ partition the $d$-cylinders, therefore we can write
$$
S=\bigcup_{i=1}^I\left(\Omega_i^I \cap S\right), \quad \Phi_{I}^t(S)=\bigcup_{i=1}^I \Phi_{I}^t\left(S \cap \Omega_i^I\right).
$$
By Lusin's theorem, one may assume that $\nabla u$ is a continuous map on $A_i \subseteq \Omega^I_i \cap [\Phi^t S]$ such that
$$
\mathcal{H}^{n-1}\left(A_i\right) \geq \mathcal{H}^{n-1}(S)-\frac{\eta}{I} .
$$
Define $A_\eta=\bigcup_{i=1}^I A_i$, which is a compact subset of $S$ and by finite additivity,
$$
\mathcal{H}^{n-1}\left(A_\eta\right)=\mathcal{H}^{n-1}\left(\bigcup_{i=1}^I A_i\right) \geq \mathcal{H}^{n-1}(S)-\eta.
$$
Let us choose dense sequences for each $\left\{a^i_k\right\}_{i=1}^I \in \Phi^{s-t} A_i$. Taking the union of all these dense sequence will give us a dense sequence in $\Phi^{s-t} A_\eta$. Approximate $d = \nabla u$ by $d_I$, then $d_I$ converges pointwise to $d= \nabla u$ on $A_\eta$. Without loss of generality, by Egorof's theorem, such convergence is uniform on each $A_i$. Since $A_\eta \subseteq \Phi^t(S)$, using the rescaling as above, we have
$$\mathcal{H}^{n-1}\left(\Phi^{s-t}_I A_\eta\right) \geq \left(\frac{\mathfrak{s}_K(s-h^-)}{\mathfrak{s}_K(t-h^-)}\right)^{n-1}\mathcal{H}^{n-1}(A_\eta).$$
By the semicontinuity of $\mathcal{H}^{n-1}$ w.r.t Hausdorff convergence, being $\Phi^{s-t}A_\eta \subseteq \Phi^s S$,
$$\limsup_{I \to \infty} \mathcal{H}^{n-1}\left(\Phi^{s-t}_I A_\eta\right) \leq \mathcal{H}^{n-1}\left(\Phi^{s-t}A_\eta\right) \leq \mathcal{H}^{n-1}\left(\Phi^sS\right).$$
Therefore we obtain
\begin{equation}\label{eq:absolutely-continuous-pushforward}
\mathcal{H}^{n-1}\left(\Phi^tS\right)\leq \left(\frac{\mathfrak{s}_K\left(t-h^{-}\right)}{\mathfrak{s}_K\left(s-h^{-}\right)}\right)^{n-1} \mathcal{H}^{n-1}\left(\Phi^s S\right).
\end{equation}
On the other hand, assume $\mathcal{H}^{n-1}\left(\Phi^sS\right) > 0$ and $h^- \leq s \leq t < h^+$. We truncate the right endpoints by considering a dense subset $\{b_i\}_i \subset \Phi^{h^+}Z$. Then we can find a similar approximation as \eqref{eq: approximation by I}:
$$\mathcal{H}^{n-1}\left(\Phi^{-t}_IS\right)\geq \left(\frac{\mathfrak{s}_K(h^+-h-t)}{\mathfrak{s}_K(h^+-h)}\right)^{n-1}\mathcal{H}^{n-1}(S).$$
Consider $A_\eta \subseteq \Phi^{s}S$ as above again and  $\Phi^{t-s}A_\eta \subseteq \Phi^t S$,
$$\mathcal{H}^{n-1}\left(\Phi^{t-s}_I A_\eta\right) \geq \left(\frac{\mathfrak{s}_K(h^+-t)}{\mathfrak{s}_K(h^+-s)}\right)^{n-1} \mathcal{H}^{n-1}(A_\eta).$$
By using the same semicontinuity of Hausdorff measure again,
$$\limsup_{I \to \infty} \mathcal{H}^{n-1}\left(\Phi^{t-s}_I A_\eta\right) \leq \mathcal{H}^{n-1}\left(\Phi^{t-s}A_\eta\right)\leq \mathcal{H}^{n-1}\left(\Phi^tS\right).$$
As a consequence, $\mathcal{H}^{n-1}\left(\Phi^sS\right) = 0$ if and only if $\mathcal{H}^{n-1}(\Phi^tS)=0$ for all $(s,t) \in (h^-,h^+)$.
\end{proof}
\begin{lemma}
The pushforward of the measure $\mathcal{H}^{n-1}\llcorner Z$ under the map $\Phi^t$ may be written as
\[
\begin{aligned}
\Phi^t_{\#}\bigl(\mathcal{H}^{n-1}\llcorner Z\bigr)
&=\alpha^t\,\bigl(\mathcal{H}^{n-1}\llcorner \Phi^t(Z)\bigr),\\
\bigl(\Phi^{-t}\bigr)_{\#}\bigl(\mathcal{H}^{n-1}\llcorner \Phi^t(Z)\bigr)
&=\frac{1}{\alpha^t\circ \Phi^t}\,\bigl(\mathcal{H}^{n-1}\llcorner Z\bigr).
\end{aligned}
\]
Moreover, if $h^-<0<h^+$, then the $\mathcal{H}^{n-1}$-measurable function $\alpha^t$ satisfies the uniform bounds
\begin{align*}
\left(\frac{\mathfrak{s}_K(h^+-t)}{\mathfrak{s}_K(h^+)}\right)^{n-1}
&\le \frac{1}{\alpha^t} \le \left(\frac{\mathfrak{s}_K(t-h^-)}{\mathfrak{s}_K(-h^-)}\right)^{n-1},
\qquad \text{for } t\ge 0,\\
\left(\frac{\mathfrak{s}_K(t-h^-)}{\mathfrak{s}_K(-h^-)}\right)^{n-1}
&\le \frac{1}{\alpha^t} \le \left(\frac{\mathfrak{s}_K(h^+-t)}{\mathfrak{s}_K(h^+)}\right)^{n-1},
\qquad \text{for } t<0.
\end{align*}
\end{lemma}
This lemma follows from Lemma~\ref{lemma: Absolutely continuous pushforward} and the Radon--Nikodym theorem.
\begin{corollary}
The map $t \mapsto \Phi^t(z)$ is invertible in $t$ and Borel in $x \in Z$. Furthermore since
$\Phi^t$ preserves $\mathcal{H}^n$-null sets in both directions, i.e., for any $N \subset\left[h^{-}, h^{+}\right] \times Z$,
$$
\mathcal{H}^n(N)=0 \quad \Longleftrightarrow \quad \mathcal{H}^n\left(\Phi^t(N)\right)=0,
$$
it also
induces an isomorphism between the spaces of $\mathcal{H}^n$-measurable functions modulo $\mathcal{H}^n$-null sets.
\end{corollary}
\begin{proof}
Let $\hat{u} = u - p$, where $p = u(z)$ for some $z \in Z$ (so $\hat{u} = 0$ on $Z$). Since $u$ is $1$-Lipschitz and satisfies $|\nabla u| = 1$ almost everywhere in $\mathcal{Z}$, the same holds for $\hat{u}$. Let $N \subset \mathcal{Z}$ be such that $\mathcal{H}^n(N) = 0$. Applying the coarea formula to $\hat{u}$ on $\mathcal{Z}$ yields:
$$0 = \int_{\mathcal{Z}}\chi_N(y) d\mathcal{H}^n(y) = \int^{h^+}_{h^-}\left\{\int_{\hat{u}^{-1}(t)} \chi_N(y) d\mathcal{H}^{n-1}(y)\right\}dt=\int^{h^+}_{h^-}\left\{\int_{\Phi^t(Z)} \chi_N(y) d\mathcal{H}^{n-1}(y)\right\}dt.$$
Recall the definition on Radon-Nikodym derivative of absolute pushforward, the above identity can be written as
\begin{align*}
\int^{h^+}_{h^-}\left\{\int_{\Phi^t(Z)} \chi_N(y) d\mathcal{H}^{n-1}(y)\right\}dt &= \int_{h^{-}}^{h^{+}}\left\{\int_{\Phi^t(Z)} \chi_{N(y)} d\left(\mathcal{H}^{n-1}\llcorner\Phi^t Z(y)\right)\right\} d t\\
&=\int_{h^-}^{h^+}\left\{\int_{\Phi^t(Z)}\chi_{N(y)}\frac{1}{\alpha^t(y)}d\left[(\Phi^t)_\#\mathcal{H}^{n-1}\llcorner Z(y)\right]\right\}dt\\
&= \int_{h^-}^{h^+}\left\{\int_{Z}\frac{1}{\alpha^t}(\Phi^t(y))d\left[\mathcal{H}^{n-1}\llcorner Z(y)\right]\right\}dt.
\end{align*}
Consequently, as $\alpha^t > 0$, for $\mathcal{H}^{n-1}$-a.e. we have that $H^{n-1}\left(\left\{y \in Z: \Phi^t(y) \in N\right\}\right)=0$.
Therefore, by Fubini's theorem,
$$\mathcal{H}^{n-1}\left((\Phi^t)^{-1}N\right) = \int_{[h^-,h^+]\times Z}\chi_{(\Phi^t)^{-1}N} d\mathcal{H}^n = \int_{h^-}^{h^+}\left\{\int_{\left\{y \in Z : \Phi^t(y) \in N\right\}}d\mathcal{H}^{n-1}\llcorner Z(y)\right\}dt = 0.$$
We conclude that $\Phi^t$ induces a Borel isomorphism between the spaces of Hausdorff measurable functions on $[h^-,h^+] \times Z$ and $\mathcal{Z}$.
\end{proof}
\begin{corollary}
The function
\[
\tilde{\alpha}(t,y):=\frac{1}{\alpha(t,y)}=\frac{\left(\Phi^{-t}\right)_{\#}\bigl(\mathcal{H}^{n-1}\llcorner\Phi^t Z\bigr)}{\mathcal{H}^{n-1}\llcorner Z}
\]
is measurable in $y$ and locally Lipschitz in $t$.
\end{corollary}
\begin{proof}
Fix $h^-<s<t<h^+$ and a measurable set $S\subset Z$. By the previous lemma,
\begin{align*}
\left(\frac{\mathfrak{s}_K(h^{+}-t)}{\mathfrak{s}_K(h^{+}-s)}\right)^{n-1}
\int_S \tilde{\alpha}(s,y)\,d\mathcal{H}^{n-1}(y)
&\le \int_S \tilde{\alpha}(t,y)\,d\mathcal{H}^{n-1}(y)\\
&\le \left(\frac{\mathfrak{s}_K(t-h^{-})}{\mathfrak{s}_K(s-h^{-})}\right)^{n-1}
\int_S \tilde{\alpha}(t,y)\,d\mathcal{H}^{n-1}(y).
\end{align*}
Since $x\mapsto \sin(\sqrt{K}x)/x$ is strictly decreasing, it follows that
$$
\frac{\sin(\sqrt{K}(t-h^-))}{\sin(\sqrt{K}(s-h^-))} \le \frac{t-h^-}{s-h^-},
\qquad
\frac{\sin(\sqrt{K}(h^+-t))}{\sin(\sqrt{K}(h^+-s))} \ge \frac{h^+-t}{h^+-s}.
$$
Let $\{t_i\}_{i\in\mathbb{N}}$ be dense in $(h^-,h^+)$. Then, for $\mathcal{H}^{n-1}$-a.e. $y\in Z$ and all $t_j\ge t_i$,
$$
\left[\left(\frac{h^{+}-t_j}{h^{+}-t_i}\right)^{n-1}-1\right]\tilde{\alpha}(t_i,y)
\le \tilde{\alpha}(t_j,y)-\tilde{\alpha}(t_i,y)
\le \left[\left(\frac{t_j-h^{-}}{t_i-h^{-}}\right)^{n-1}-1\right]\tilde{\alpha}(t_i,y).
$$
Finally, Bernoulli's inequality $(1-x)^{n-1}\ge 1-(n-1)x$ for $x\in[0,1]$ yields a Lipschitz bound for $\tilde{\alpha}$.
\end{proof}
\begin{lemma}\label{lemma: local disintegration}
On $\mathcal{K}$, we have the following disintegration of the Riemannian volume measure: for every $\varphi \in L^1(M, \operatorname{Vol}_g)$,
$$
\int_{\mathcal{K}} \varphi(x)\, d\operatorname{Vol}_g(x)
= \int_Z \left\{\int_{h^-}^{h^+} \varphi\bigl(\Phi^t(y)\bigr)\,\tilde{\alpha}(t,y)\, dt\right\} d\mathcal{H}^{n-1}(y).
$$
\end{lemma}
The proof of this local disintegration on the $\mathcal{K}$-cylinder follows essentially the same lines as in \cite{caravenna2011a}.

\begin{maintheorem}[Main Theorem]\label{thm:main1}
Let $(M, g)$ be a Riemannian manifold with positive sectional curvature, and let $\mu \ll \operatorname{Vol}_g$ and $\nu$ be the source and target measures, respectively. We define the transport set $\mathcal{T}_{\mathfrak{e}}$ based on $\operatorname{spt}(\mu)$ and $\operatorname{spt}(\nu)$ as previously described. On $\mathcal{T}_{\mathfrak{e}}$, the following disintegration of the Riemannian volume measure (which coincides with the Hausdorff measure up to normalization) holds:
\[
\int_{\mathcal{T}_{\mathfrak{e}}} \varphi(x)\, d\operatorname{Vol}_g(x)
= \int_{\mathcal{S}}\left\{\int_{u(a(y))}^{u(b(y))} \varphi\left(\exp_y\left[(t-u(y))\nabla u(y)\right]\right)D(t,y)\, dt\right\} d\mathcal{H}^{n-1}(y), \quad \varphi \in C^\infty_c(M).
\]
where
\begin{itemize}
    \item $\{\mathcal{Z}_q\}_{q\in \mathbb{Q}}$ is the partition of the transport set $\mathcal{T}_{\mathfrak{e}}$ into sheaf sets.
    \item $Z_q$ is the relative base of $\mathcal{Z}_q$ such that $u(y) = q$ for all $y \in Z_q$.
    \item $\mathcal{S}$ is the quotient set of $\mathcal{T}_{\mathfrak{e}}$ with respect to membership in transport rays: $\mathcal{S} = \bigcup_{q \in \mathbb{Q}} Z_q$.
    \item The density function is defined by $D(t,y) := \sum_{q} \tilde{\alpha}(t-u(y),y)\,\chi_{\mathcal{Z}_q}(y)$.
\end{itemize}
\end{maintheorem}
\begin{proof}
Since $\operatorname{Vol}_g\bigl(\mathcal{T}_{\mathfrak{e}}\setminus \mathcal{T}\bigr)=0$, we may integrate over $\mathcal{T}$. Consider a refinement of the partition $\{\mathcal{Z}_q\}$, which partitions $\mathcal{T}$ into cylinders $\{\mathcal{K}_{qi}\}_{i \in \mathbb{N}}$ with base $\hat{Z}_{qi}$ and truncated lengths $h^{\pm}_{qi}$. The quotient set is $Z_q = \bigcup_{i \in \mathbb{N}} \hat{Z}_{qi}$, and
\[
\mathcal{K}_{qi} = \left\{\Phi^t\left(\hat{Z}_{qi}\right) : t \in [h^-_{qi},h^+_{qi}]\right\}
= \left\{y: u(z)-u(y) \in [h^-_{qi},h^+_{qi}]\right\} \cap \bigcup_{z \in \hat{Z}_{qi}} \llbracket a(z)b(z)\rrbracket.
\]
By Lemma \ref{lemma: local disintegration},
\begin{align*}
&\int_{\mathcal{K}_{qi}} \varphi(x) d\mathcal{H}^n(x) = \int_{\hat{Z}_{qi}}\left\{\int^{h^+_{qi}}_{h^-_{qi}}\varphi\left(\Phi^t(y)\right)\tilde{\alpha}(t,y) d\mathcal{H}^1(t)\right\}d\mathcal{H}^{n-1}(y)\\
&=\int_{\hat{Z}_{qi}}\left\{\int^{h^+_{qi}}_{h^-_{qi}}\varphi\left(\Phi^t(y)\right)\tilde{\alpha}(t,y) d\mathcal{H}^1(t)\right\}d\mathcal{H}^{n-1}(y)\\
&=\int_{\hat{Z}_{qi}}\left\{\int^{h^+_{qi}}_{h^-_{qi}} \varphi\left(\exp_y\left[t\nabla u(y)\right]\right)\tilde{\alpha}(t,y) d\mathcal{H}^1(t)\right\}d\mathcal{H}^{n-1}(y)\\
&=\int_{\hat{Z}_{qi}}\left\{\int^{h^+_{qi}+u(y)}_{h^-_{qi}+u(y)} \varphi\left(\exp_y\left[(t-u(y))\nabla u(y)\right]\right)\tilde{\alpha}(t-u(y),y) d\mathcal{H}^1(t)\right\}d\mathcal{H}^{n-1}(y)\\
&=\int_{\hat{Z}_{qi}}\left\{\int^{h^+_{qi}+u(y)}_{h^-_{qi}+u(y)} \varphi\left(\exp_y\left[(t-u(y))\nabla u(y)\right]\right)D(t,y)d\mathcal{H}^1(t)\right\}d\mathcal{H}^{n-1}(y).
\end{align*}
By the countable additivity, one can extend the result on $\mathcal{T}$:
\begin{align*}
&\int_{\mathcal{T}} \varphi(x) d \mathcal{H}^n(x)=\int_{\cup_{qk} \mathcal{K}_{qk}} \varphi(x) d \mathcal{H}^n(x)=\sum_{qk} \int_{\mathcal{K}_{qk}} \varphi(x) d \mathcal{H}^n(x)\\
&=\sum_{q}\sum_k \int_{\hat{Z}_{qk}}\left\{\int^{h^+_{qk}+u(y)}_{h^-_{qk}+u(y)}\varphi\left(\exp_y\left[(t-u(y))\nabla u(y)\right]\right)D(t,y)d\mathcal{H}^1(t)\right\}d\mathcal{H}^{n-1}(y)\\
&= \sum_q \int_{Z_q}\int^{u(b(y))}_{u(a(y))}\varphi\left(\exp_y\left[(t-u(y))\nabla u(y)\right]\right)D(t,y)d\mathcal{H}^1(t)d\mathcal{H}^{n-1}(y).
\end{align*}
To justify the interchange of the countable sum and the integral, we use the Monotone Convergence Theorem. Write $\varphi=\varphi^{+}-\varphi^{-}$ and apply MCT to each part separately. Since $\varphi^{+}$ and $\varphi^{-}$ are non-negative,
\[
\sum_q \sum_k \int_{\mathcal{K}_{qk}} \varphi^{+}\, d\operatorname{Vol}_g
=\int_{\mathcal{T}_{\mathfrak{e}}} \varphi^{+}\, d\operatorname{Vol}_g,
\]
(and similarly for $\varphi^{-}$). Moreover, each integral $\int_{\mathcal{K}_{qk}} \varphi^{+}\, d\operatorname{Vol}_g$ is finite (since $M$ is compact), hence the limit of the partial sums is finite and well-defined.
\end{proof}
\section{Regularity of the divergence}
Since the transport data are supported on $\mathcal{T}$, we extend the vector field $d:=\nabla u$ to all of $M$ by setting $d=0$ on $M\setminus \mathcal{T}$. Furthermore, if $a\in\mathcal{T}$, then any point of the cut locus $\operatorname{Cut}(a)$ is either an endpoint (i.e. it lies in $\mathcal{T}_{\mathfrak{e}}\setminus\mathcal{T}$) or it lies outside the transport set (i.e. in $M\setminus\mathcal{T}_{\mathfrak{e}}$).
\begin{defn}
Let $\mathcal{K}= \{\Phi^t(Z): t \in [h^-,h^+]\}$ be a d-cylinder and assume $Z$ is compact. Suppose, moreover, that for $\operatorname{Vol}_g$-a.e.$x \in \mathcal{K}$ the ray $\mathcal{R}(x)$ intersects also the compact $K = \Phi^{h^{-}-\varepsilon}(Z)$ for some $\varepsilon > 0$. Let $\{a_i\}_{i \in \mathbb{N}}$ be a dense set in $K$. Consider the potential given by
$$\hat{u}(x) = \max\{u(a_i) - d(x,a_i): a_i \in K\}.$$
\end{defn}
\begin{remark}
It is not hard to see that $\hat{u}|_{\mathcal{K}} = u|_{\mathcal{K}}$. Note that we can write
\[
 u|_{\mathcal{K}}(x) = \max \left\{u(b) - d(x,b): b \in K\right\}.
\]
Let $y \in K$ with $y \in \mathcal{R}(x)$. Then there exists $a \in \{a_i\}_{i \in \mathbb{N}}$ such that $d(a,y) < \varepsilon/2$. We have
\begin{align*}
\hat{u}(x) &= \max_{a_i \in K}\left\{u(a_i) - d(x,a_i) : a_i \in K\right\}\geq u(a) - d(x,a)\\
&= u(a) - \bigl(d(x,y) + d(y,a)\bigr)\\
&\geq u(a) - d(x,y) - \frac{\varepsilon}{2}\\
&= u(a) - u(y) + u(y) - d(x,y) - \frac{\varepsilon}{2}\\
&\geq -\varepsilon + u(y) - d(x,y).
\end{align*}
Since $\varepsilon > 0$ is arbitrary, we conclude that $\hat{u}|_{\mathcal{K}} = u|_{\mathcal{K}}$.
\end{remark}
We aim to define $\hat{d} : M \backslash K\to TM$ to be the relative vector field of rays $\hat{u}$ which coincides with $d = \nabla u$ on $\mathcal{K}_i$. To do this, we will use the modified Voronoi tessellation again.
\begin{lemma}
On $(M,g)$ with $\operatorname{Sec}(M) \geq k > 0$, the vector field $\hat{d}$ is single valued on $\mathcal{H}^n$-a.e. on $M$. Moreover, its divergence $\operatorname{div}\hat{d}$ is a signed Radon measure on $M\backslash K$.
\end{lemma}
\begin{proof}
First notice that the initial points $K = \Phi^{h^{-}-\varepsilon}(Z)$ is $\mathcal{H}^{n-1}$-rectifiable, hence the vector field $\hat{d}$ is $\mathcal{H}^n$-a.e. well-defined. The regularity of the divergence, which in general should be only
a distribution, is now proved by approximation. As before, the potentials
$$\hat{u}_I(x) = \max\left\{u(a_j)-d(x,a_j): j = 1,\dots,I\right\}$$
converges uniformly to $\hat{u}$. Moreover, the corresponding vector field of directions
$$d_I(x) = \sum_{i=1}^I d^i(x) \chi_{\Omega_i^I}(x),\quad d^i(x) = -\nabla_x d(a_i,x)$$
with modified Voronoi cells
$$\Omega_i^I = \left\{x \in M : u(a_i)-d(x,a_i) \geq u(a_j) - d(x,a_j), j \in \{1,\dots, I\}\backslash i\right\},$$
and their $\mathcal{H}^{n-1}$-rectifiable boundaries:
$$
\begin{aligned}
J_I & =\bigcup_{i \neq j}\left(\bar{\Omega}_i \cap \bar{\Omega}_j\right) \\
& =\left\{x: \exists i, j, i \neq j, \phi\left(a_i\right)-d(a_i,x)=\phi\left(a_j\right)-d(a_j,x)\right\}.
\end{aligned}
$$
We have shown that $d_I= \nabla \hat{u}_I$ converges point-wise $\mathcal{H}^n$-a.e. to $\hat{d} = \nabla \hat{u}$. Next we show that $\operatorname{div}(d_I)$ converges to $\operatorname{div}(\hat{d})$ in the sense of distribution on any compact $\Omega \subset M \backslash K$. Since $\{d_I\}\to \hat{d}$ pointwise $\mathcal{H}^n$-a.e,
$$\lim_{I\to \infty}\int_{\Omega} \varphi d(\operatorname{div}(d_I)) = \lim_{I\to\infty} -\int_\Omega\left\langle\nabla \phi, d_I\right\rangle d \mathcal{H}^n,\quad \forall \varphi \in C^\infty_c(\Omega).$$
Since $\varphi$ is smooth and compactly supported, we can use dominated convergence theorem to pass the limit under the integral sign,
$$
\lim _{I \rightarrow \infty}\left(-\int_{\Omega}\left\langle\nabla \varphi, d_I\right\rangle d\mathcal{H}^n\right)=-\int_{\Omega}\langle\nabla \varphi, \hat{d}\rangle d \mathcal{H}^n.
$$
Therefore $\operatorname{div}(d_I) \rightarrow \operatorname{div}(\hat{d})$ in $\mathcal{D}'(\Omega)$, where $\Omega$ is any compact subset of $M \backslash K$. Since $d_I$ are functions of bounded variation, $\{\operatorname{div}(d_I)\}$ is a sequence of Radon measures. We need to show that $\operatorname{div}(\hat{d})$ is a Radon measure. Notice that $-d(a_i,\cdot)$ are semi-convex functions, hence $\{\hat{u}_I\}$ is a sequence of uniformly semi-convex functions. Therefore $\operatorname{div}(d_I) = \Delta \hat{u}_I$ is uniformly bounded below by a negative constant 
$$\operatorname{div}(d_I)(x) \geq -C_\Omega,\quad \text{for all compact subset $\Omega$ of } M \backslash K.$$
Since this bound is uniform for all $I$, the limit distribution must satisfy the same bound. Hence the distribution $T|_{\Omega}= C_\Omega+\operatorname{div}(\hat{d})$ is positive. Therefore $T$ is a Radon measure. Since $C_\Omega$ is simply a constant, we conclude that $\operatorname{div}(\hat{d})$ is also a Radon measure.
\end{proof}
\begin{remark}
(Singularity of the distributional divergence). Since $d(a_i,\cdot)$ are smooth functions on $\Omega_i^I \cap (M \backslash K)$, their gradients are smooth as well, and $d^i \in \operatorname{BV}\left(\Omega_i^I \cap (M \backslash K); T_xM \right)$. Furthermore, it is not hard to show that each $\Omega_i^I$ is a set of finite perimeter; hence $\chi_{\Omega_i^I} \in \operatorname{BV}(\Omega_i^I)$. Therefore $d_I$ is indeed a function of bounded variation on $M$, outside a $\mathcal{H}^n$-negligible set. The distribution $\operatorname{div} d_I$ is a locally finite Radon measure. By the Riesz representation theorem,
$$
\left\langle\operatorname{div} d_I, \varphi\right\rangle=-\int_M \left\langle\nabla \varphi, d_I\right\rangle_g \, d \operatorname{Vol}_g=\int_M \varphi\, d(\operatorname{div} d_I) \quad \forall \varphi \in C_{\mathrm{c}}^{\infty}(M).
$$
We decompose $\operatorname{div} d_I$ as
\begin{align*}
d(\operatorname{div} d_I)(x) &= d(\operatorname{div} d_I)_{\text{a.c.}}(x) + d(\operatorname{div} d_I)_{\text{jump}}(x)\\
&= \sum_{i=1}^I -\operatorname{Tr}(\nabla d^i(x))\, d\operatorname{Vol}_g \llcorner \Omega^I_i(x) + \sum_{i,j=1}^I \langle d^i(x) - d^j(x), \nu_{ij}(x)\rangle \, d\mathcal{H}^{n-1}\llcorner J_{ij}(x)\\
&= \sum_{i=1}^I -\Delta_g d(a_i,x)\, d\operatorname{Vol}_g \llcorner \Omega^I_i(x) + \sum_{i,j=1}^I \langle d^i(x) - d^j(x), \nu_{ij}(x)\rangle \, d\mathcal{H}^{n-1}\llcorner J_{ij}(x).
\end{align*}
Observe that $\Delta_g d(a_i,x)$ is not positive in general. By the Laplacian comparison theorem \cite{petersen2006riemannian},
$$-\Delta_g d(a_i,x) \geq-\frac{n-1}{d\left(x, \cup_i^I a_i\right)}.$$
Since $M$ is a compact Riemannian manifold, the sectional curvature is a smooth function. On $\Omega_i^I \cap (M \backslash K)$, the sectional curvature is bounded from above by a constant $P$ such that $\sqrt{P} \in \mathbb{N}$. By the Laplacian comparison theorem again,
$$
\Delta d(a_i,x) \geq(n-1) \sqrt{P} \cot \left(\sqrt{P} d(a_i,x) \right).
$$
Therefore, for all $x \in \Omega_i^I \cap (M \backslash K)$,
\begin{align*}
\bigl|d(\operatorname{div} d_I)_{\text{a.c.}}\bigr|(x)&\leq \max\left\{ \frac{n-1}{d\left(x, \cup_i a_i\right)}, (n-1) \sqrt{P} \left|\cot\left(\sqrt{P} \max_{i=1,\dots,I}(d(x,a_i))\right)\right|\right\}\\
&\leq (n-1)\max\left\{\frac{1}{d(x, \cup_{i}^I a_i)}, \frac{1}{|\pi/\sqrt{P}-\max_{i=1,\dots, I} d(x,a_i)|}\right\}.
\end{align*}
Let $\rho_\delta:[0, \infty) \rightarrow[0, \infty)$ be a standard mollifier kernel supported in $[0, \delta]$, scaled so that $$\int_M \rho_\delta\left(d(y, z)\right) \, d \operatorname{Vol}_g(z)=1.$$ Define the smoothed indicator function of the geodesic ball $B(x, r)$ by
$$
\varphi_\delta(y)=\int_M \chi_{B(x, r)}(z) \, \rho_\delta\left(d(y, z)\right) \, d \operatorname{Vol}_g(z)
=\int_{B(x, r)} \rho_\delta\left(d(y, z)\right) \, d \operatorname{Vol}_g(z).
$$
For small enough $\delta$, away from the conjugate point, $\varphi_\delta$ is a smooth function supported in $B(x, r+\delta)$. By the divergence theorem,
\begin{align*}
\nabla_y \varphi_\delta(y) &=\int_{B(x, r)} \nabla_y\left[\rho_\delta\left(d(y, z)\right)\right] \, d \operatorname{Vol}_g(z) = \int_{B(x, r)}-\nabla_z\left[\rho_\delta\left(d(y, z)\right)\right] \, d \operatorname{Vol}_g(z)\\
&=-\int_{\partial B(x, r)} \rho_\delta\left(d_g(y, z)\right) \, \nu(z) \, d \mathcal{H}^{n-1}(z),
\end{align*}
where $\nu(z)$ is the outward unit normal to $\partial B(x, r)$ at $z$. By taking limits of both sides and using the dominated convergence theorem,
\begin{align*}
\lim _{\delta \rightarrow 0^{+}}\left\langle\operatorname{div} d_I, \varphi_\delta\right\rangle &=\lim _{\delta \rightarrow 0^{+}} \int_{\partial B(x, r)}\left(\int_{M} d(y) \rho_\delta(y-z) d y\right) \cdot \nu(z) d \Omega(z)\\
&=\int_{\partial B(x, r)} d(y) \cdot \nu(y) d \Omega(y).
\end{align*}
Therefore,
$$(\operatorname{div} d_I)_{\text{jump}}(B(r))=(\operatorname{div} d_I)_{\text{jump}}^{+}(B(r))-(\operatorname{div} d_I)_{\text{jump}}^{-}(B(r))\geq-|\partial B(x, r)|.$$
Since the total variation measure satisfies $|\operatorname{div} d_I| = |(\operatorname{div} d_I)_{\text{a.c.}}| + |(\operatorname{div} d_I)_{\text{jump}}|$, we apply the triangle inequality to the equation above:
\begin{align*}
|\operatorname{div} d_I|(B_r) &= \int_{B_r(x)} |(\operatorname{div} d_I)_{\text{a.c.}}| + |(\operatorname{div} d_I)_{\text{jump}}| \\
&\leq \int_{B_r(x)} |(\operatorname{div} d_I)_{\text{a.c.}}| + \left( \left| \int_{\partial B_r(x)} \langle d_I, \nu \rangle \right| + \int_{B_r(x)} |(\operatorname{div} d_I)_{\text{a.c.}}| \right) \\
&\leq 2\int_{B_r(x)} |(\operatorname{div} d_I)_{\text{a.c.}}|  \, d\operatorname{Vol}_g + |\partial B_r(x)|.
\end{align*}
Since the singularity of the a.c. part is of order $O(1/d)$, it is integrable in dimension $n \geq 2$. Therefore, the total variation on the ball is finite and controlled by the boundary area and the integrated curvature bounds.
\end{remark}
\begin{lemma}
Let $\mathcal{K}$ be the cylinder set fixed above for defining $\hat{d}=\nabla \hat{u}$. Consider $(\mathcal{S},c)$ as in the main theorem~\ref{thm:main1}, with respect to the transport direction~$\hat{d}$. Then, for any sub-cylinder $\mathcal{K}'\subset\mathcal{K}$, the following continuity equation holds:
\begin{equation}\label{eq:continuity-equation}
\partial_t D(t,y)-\bigl[(\operatorname{div}\hat{d})_{\mathrm{a.c.}}\bigl(\exp_y\bigl((t-u(y))\hat d(y)\bigr)\bigr)\bigr]\,D(t,y)=0
\quad \mathcal{H}^n\text{-a.e.\ on }\mathcal{K}.
\end{equation}
and the Green--Gauss formula
\begin{equation}
\int_{\mathcal{K}'} g(\nabla\varphi,\hat d)\,d\mathcal{H}^n
=-\int_{\mathcal{K}'} \varphi\,(\operatorname{div}\hat d)_{\mathrm{a.c.}}\,d\mathcal{H}^n
+\int_{\partial\mathcal{K}'} \varphi\, g(\hat d,\mathbf{n}_{\mathcal{K}'})\, d\mathcal{H}^{n-1}.
\end{equation}
\end{lemma}
\begin{proof}
By the previous lemma, $\operatorname{div}\hat{d}$ is a Radon measure, hence
$$
-\int_{M}\nabla\varphi\cdot\hat{d}=\langle\operatorname{div}\hat{d},\varphi\rangle
=\int_{M}(\operatorname{div}\hat{d})_{\mathrm{a.c.}}\,\varphi+\int_{M}\varphi\,(\operatorname{div}\hat{d})_{\mathrm{s}}
\quad \forall \varphi\in C_c^{\infty}(M\backslash K).
$$
Moreover, $\hat{d} = \nabla \hat{u}$ is the direction of transportation with respect to $\hat{u}$. We explicitly decompose the measure along $\hat{u}$, according to the main theorem \ref{thm:main1}. Since $\hat{u}|_{\mathcal{K}} = u|_{K}$, we have $\widehat{\mathcal{S}}|_{\mathcal{K}} \equiv \mathcal{S}|_{\mathcal{K}}$ and $\hat{D}(t,y)|_{\widehat{\mathcal{S}}} = D(t,y)|_{\mathcal{S}}$. Since $c$ is local lipschitz in $t$ and we are integrating over a compact set, we can perform integration by parts:
\begin{align*}
&\int^{\hat{u}(\hat{b}(y))}_{\hat{u}(\hat{a}(y))}\hat{D}(t,y)\nabla \varphi\left(\exp_y\left([t-u(y)]\nabla u(y)\right)\right)\cdot\hat{d} dt = \varphi(\hat{b}(y))\hat{D}(\hat{u}(\hat{b}(y)),y) - \varphi(\hat{a}(y))\hat{D}(\hat{u}(\hat{a}(y)),y)\\
&-\int^{\hat{u}(\hat{b}(y))}_{\hat{u}(\hat{a}(y))} \varphi\left(\exp_y\left([t-u(y)]\nabla u(y)\right)\right)  \partial_t \hat{D}(t, y) d t.
\end{align*}
\noindent
Let $\widehat{\Phi}(t,y):=\exp_y\bigl((t-\hat{u}(y))\hat{d}\bigr)$. Note that $\partial_t \widehat{\Phi}(t,y)=\hat{d}(\widehat{\Phi}(t,y))$.
Integrating the identity above with respect to $y\in\widehat{\mathcal{S}}$ (with the measure $d\mathcal{H}^{n-1}$ from the disintegration in Theorem~\ref{thm:main1}) yields
\begin{align*}
0=&\int_{\widehat{\mathcal{S}}} \varphi(\hat{b}(y))\, \hat{D}(\hat{u}(\hat{b}(y)),y)\, d \mathcal{H}^{n-1}(y)
-\int_{\widehat{\mathcal{S}}} \varphi(\hat{a}(y))\, \hat{D}(\hat{u}(\hat{a}(y)),y)\, d \mathcal{H}^{n-1}(y)\\
&-\int_{\widehat{\mathcal{S}}} \int_{\hat{u}(\hat{a}(y))}^{\hat{u}(\hat{b}(y))}
\varphi\bigl(\widehat{\Phi}(t,y)\bigr)\,\partial_t \hat{D}(t, y)\, d t \, d \mathcal{H}^{n-1}(y)\\
&+\int_{\widehat{\mathcal{S}}} \int_{\hat{u}(\hat{a}(y))}^{\hat{u}(\hat{b}(y))}
(\operatorname{div} \hat{d})_{\mathrm{a.c.}}\bigl(\widehat{\Phi}(t,y)\bigr)\,\varphi\bigl(\widehat{\Phi}(t,y)\bigr)\,\hat{D}(t, y)\, d t \, d \mathcal{H}^{n-1}(y)\\
&+\int_{M} \varphi\, d\bigl((\operatorname{div}\hat{d})_{\mathrm{s}}\bigr).
\end{align*}
Moreover, since both $\hat{D}$ and $\partial_t \hat{D}$ are locally bounded, the dominated convergence theorem implies that the last identity also holds for bounded test functions that vanish outside a compact set and in a neighborhood of $K$, the set of initial points for $\hat{d}$. By the arbitrariness of $\varphi$, this yields \eqref{eq:continuity-equation} for $\mathcal{H}^n$-a.e.\ point in $\mathcal{K}$.

Recall that by our definition of $\hat{d}$, the singular part is concentrated on $\bigcup_{y \in K}\bigl(\hat{a}(y)\cup \hat{b}(y)\bigr)$ with respect to the ray defined by $\hat{u}$. Let us fix a cylinder $\mathcal{K}$ with base $\hat{Z}$. Define $\widehat{\Phi}^\pm$ to be the maps that associate to each point in $\hat{Z}$ its corresponding endpoints, i.e.\ the maps $\hat{Z}\ni z \mapsto \hat{a}(z)$ and $\hat{Z}\ni z \mapsto \hat{b}(z)$. It is then clear that the singular part is concentrated on
\[
\hat{D}\cdot\widehat{\Phi}^+_{\#}\,\mathcal{H}^{n-1}\!\llcorner \hat{Z}
-\hat{D}\cdot\widehat{\Phi}^-_{\#}\,\mathcal{H}^{n-1}\!\llcorner \hat{Z}.
\]
Hence the singular part of $\operatorname{div}\hat{d}$ vanishes on $\mathcal{K}'$. Therefore, we have
\[
-\int_{\hat{Z}}\int_{h^-}^{h^+} \varphi\bigl(\exp_y((t-u(y))\hat{d}(y))\bigr)\,\partial_t D(t,y)\,dt\, d\mathcal{H}^{n-1}(y)
\]
\[
+ \int_{\hat{Z}}\int_{h^-}^{h^+} (\operatorname{div}\hat{d})_{\mathrm{a.c.}}\bigl(\exp_y((t-u(y))\hat{d}(y))\bigr)\,\varphi\bigl(\exp_y((t-u(y))\hat{d}(y))\bigr)\,D(t,y)\,dt\, d\mathcal{H}^{n-1}(y)=0.
\]
Integrating by parts in the first term (with respect to $t$), we obtain
\begin{align*}
0={}&\int_{\hat{Z}}\bigl[\varphi\bigl(\exp_y((t-u(y))\hat{d}(y))\bigr)\,D(t,y)\bigr]_{t=h^-}^{t=h^+}\, d\mathcal{H}^{n-1}(y)\\
&+\int_{\hat{Z}}\int_{h^-}^{h^+} g\Bigl(\nabla \varphi\bigl(\exp_y((t-u(y))\hat{d}(y))\bigr),\,\hat{d}\bigl(\exp_y((t-u(y))\hat{d}(y))\bigr)\Bigr)\,D(t,y)\,dt\, d\mathcal{H}^{n-1}(y)\\
&+\int_{\hat{Z}}\int_{h^-}^{h^+} \varphi\bigl(\exp_y((t-u(y))\hat{d}(y))\bigr)\,(\operatorname{div}\hat{d})_{\mathrm{a.c.}}\bigl(\exp_y((t-u(y))\hat{d}(y))\bigr)\,D(t,y)\,dt\, d\mathcal{H}^{n-1}(y).
\end{align*}
If we choose test functions of the form $\varphi\,\chi_{\mathcal{K}'}$ with $\varphi\in C_c^\infty(M\setminus K)$ and $\mathcal{K}'\Subset\mathcal{K}$, the distributional Leibniz rule yields
\[
\nabla(\varphi\,\chi_{\mathcal{K}'})
=\chi_{\mathcal{K}'}\,\nabla\varphi+\varphi\,\nabla\chi_{\mathcal{K}'},
\]
where $\nabla\chi_{\mathcal{K}'}$ is the outward normal measure on $\partial\mathcal{K}'$, i.e.\ $\nabla\chi_{\mathcal{K}'}=\mathbf{n}_{\mathcal{K}'}\,\mathcal{H}^{n-1}\!\llcorner\partial\mathcal{K}'$.
Consequently, we obtain the Green--Gauss formula on $\mathcal{K}'$:
\[
\int_{\mathcal{K}'} g(\nabla\varphi,\hat d)\,d\mathcal{H}^n
=-\int_{\mathcal{K}'} \varphi\,(\operatorname{div}\hat d)_{\mathrm{a.c.}}\,d\mathcal{H}^n
+\int_{\partial\mathcal{K}'} \varphi\, g(\hat d,\mathbf{n}_{\mathcal{K}'})\, d\mathcal{H}^{n-1}.
\]
\end{proof}
\begin{remark}
From the Lagrangian perspective, we may regard $\Phi^t$ as a flow map that moves particles along the transport rays. The vector field $\nabla u$ is the velocity field and therefore satisfies the ODE
\[
\frac{d}{dt}\Phi^s(y)=\nabla u\bigl(\Phi^s(y)\bigr).
\]
Therefore, the pair $(c,\nabla u)$ solves the continuity equation uniquely~\cite{santambrogio2015optimal}. In fact, we will see in the next lemma that the absolutely continuous part is independent of the partition $\{\mathcal{K}\}$.
\end{remark}

We glue together the distributional divergence of $\hat{d}$, which coincides with $d$ on each cylinder; this formally yields a continuity equation defined on the $\mathcal{H}^n$-a.e. transport set.
\begin{lemma}
Formally, we glue together the absolutely continuous parts of the distributional divergence on $\mathcal{T}$ by setting
\[
(\operatorname{div} d)_{\mathrm{a.c.}}
  =\sum_{j}\bigl((\operatorname{div} d)_{\mathrm{a.c.}}\bigr)_j\,\chi_{\mathcal{K}_j}.
\]
Then, for any partition into cylinders as in Theorem~\ref{thm:main1}, with relative density $c$ and sections $\mathcal{S}$, one has
\[
\partial_t D(t,y)
-\Bigl[(\operatorname{div}d)_{\mathrm{a.c.}}\bigl(\exp_y((t-u(y)) d(y))\bigr)\Bigr]\,D(t,y)
=0
\quad \mathcal{H}^n\text{-a.e.\ on } z \in \mathcal{T}.
\]
where $z = \exp_y\left((t-u(y)\nabla u(y))\right)$.
\end{lemma}
Under the assumptions of this paper, $\operatorname{div}d$ is a finite sum of Radon measures; therefore, it is a Radon measure.

\begin{remark}[Open question]
The main theorem~\ref{thm:main1} yields a well-defined volume measure on the transport set $\mathcal{T}$ via the relative density $D(t,y)$ in the disintegration. It would be interesting to clarify how this volume is related to the optimal $L^1$-transport cost
\[
\operatorname{KP}[\pi]=\inf_{\pi\in\Pi(\mu,\nu)}\int_{M\times M} d(x,y)\,d\pi(x,y) = \sup_{u\in\operatorname{Lip_1}(M,d)} \int_M ud(\mu-\nu).
\] 
More precisely, can one choose a Kantorovich potential $u\in\operatorname{Lip}_1(M,d)$, that is a u attaining the supremum in $KP$ such that for all test functions $\phi$ we have 
$$
\int_{\mathcal{T}_e} \varphi(x) d \operatorname{Vol}_g(x)=\int_{\mathcal{S}}\left(\int_{u(b(y))}^{u(a(y))} \varphi\left(\exp _y[(t-u(y)) \nabla u(y)]\right) D(t, y) d t\right) d \mathcal{H}^{n-1}(y).
$$
This would be saying that the resulting decomposition identifies the transport density and expresses the total cost in terms of the induced volume on $\mathcal{T}$.
\end{remark}
\nocite{*} 
\clearpage
\addcontentsline{toc}{section}{References}
\bibliographystyle{amsalpha}
\bibliography{references} 
\noindent\textsc{Zhengyao Huang}, Department of Mathematics, Durham University, UK\\
\textit{Email address}: \texttt{phqv76@durham.ac.uk}
\end{document}